\documentclass{amsart}
\usepackage{amsfonts,amssymb,amscd,amsmath,enumerate,verbatim,calc}

\newcommand{\CM}{Cohen-Macaulay}

\newcommand{\wrt}{with respect to}

\newcommand{\n}{\mathfrak{n} }
\newcommand{\m}{\mathfrak{m} }
\newcommand{\M}{\mathfrak{M} }

\newcommand{\A}{\mathfrak{a} }
\newcommand{\R}{\mathcal{R} }

\newcommand{\C}{\mathcal{C}}
\newcommand{\rt}{\rightarrow}

\newcommand{\ov}{\overline}

\newcommand{\Hom}{\operatorname{Hom}}

\theoremstyle{plain}

\newtheorem{theorem}{Theorem}[section]

\newtheorem{lemma}[theorem]{Lemma}

\theoremstyle{definition}
\newtheorem{definition}[theorem]{Definition}
\newtheorem{s}[theorem]{}
\newtheorem{remark}[theorem]{Remark}

\theoremstyle{remark}

\begin{document}

\title{On $p_g$-ideals in positive characteristic}
\author{Tony~J.~Puthenpurakal}
\date{\today}
\address{Department of Mathematics, IIT Bombay, Powai, Mumbai 400 076}

\email{tputhen@math.iitb.ac.in}

 \begin{abstract}
Let $(A,\m)$ be an excellent normal domain of dimension two containing a  field $k \cong A/\m$. An $\m$-primary ideal $I$ to be  a $p_g$-ideal if the Rees algebra $A[It]$ is a \CM \ normal domain.
If $k$ is algebraically closed then Okuma, Watanabe and  Yoshida proved that $A$ has $p_g$-ideals and furthermore product of two
$p_g$-ideals is a $p_g$ ideal. In a previous paper we showed that if $k$ has characteristic zero then $A$
has $p_g$-ideals. In this paper we prove that if $k$ is perfect field of positive characteristic then also $A$ has $p_g$ ideals.
\end{abstract}
 \maketitle
\section{introduction}
Dear Reader, while reading this paper it is a good idea to have \cite{P} nearby.
Let $(A,\m)$ be a normal domain of dimension two.
By Zariski's theory of integrally closed ideals in a two dimensional regular local ring, we get that if $A$ is regular and $I$ is an integrally closed $\m$-primary ideal then the Rees algebra $\R(I) = A[It]$ is a \CM \ normal domain;
see \cite[Chapter 14]{HS} for a modern exposition.
Later
Lipman proved that if $(A,\m)$ is a two dimensional rational singularity then  analogous results holds, see \cite{Lipman}.

Assume $(A,\m)$ is an excellent normal domain of dimension two  containing an algebraically closed field $k \cong A/\m$.
For such rings Okuma, Watanabe and  Yoshida in \cite{OWY-1} introduced (using geometric techniques) the notion of $p_g$-ideals as follows:
let $I$ be an $\m$-primary ideal in $A$.
The $I$ has a resolution $f \colon X \rt Spec(A)$ with $I\mathcal{O}_X $ invertible. Then $I\mathcal{O}_X = \mathcal{O}_X(-Z)$ for some anti-nef
cycle $Z$.   It can be shown that $\ell_A(H^1(X, \mathcal{O}_X(-Z)) \leq p_g(A)$ where $p_g(A) = \ell_A(H^1(X,\mathcal{O}_X)$ is the geometric genus of $A$ and $Z$ is an anti-nef cycle such that $\mathcal{O}_X(-Z)$ has no fixed component.
 An integrally closed $\m$-primary ideal $I$ with $\ell_A(H^1(X, \mathcal{O}_X(-Z)) = p_g(A)$  is called a $p_g$-ideal. They showed that $p_g$ ideals exist in $A$. Furthermore
if $I, J$ are two $\m$-primary $p_g$ ideals then they also proved that $IJ$ is a $p_g$-ideal. Furthermore $I$ is stable and so  the Rees algebra $\R(I)$ is a \CM \ normal domain.
They also proved that if $A$ is also a rational singularity then any $\m$-primary integrally closed ideal
is a $p_g$-ideal. In a later paper \cite{OWY-2} they showed that if $\R(I)$ is a \CM \ normal domain then $I$ is a $p_g$-ideal.

Motivated by this result  we made the following definition in \cite{P}:
\begin{definition}
 Let $(A, \m)$ be a normal domain of dimension two. An $\m$-primary ideal $I$ is said to be $p_g$-ideal in $A$ if the Rees algebra
 $\R(I) = A[It]$ is a normal \CM \  domain.
\end{definition}
We note that if $I$ is a $p_g$-ideal then all powers of $I$ are integrally closed. Furthermore if the residue field of $A$ is infinite then $I$ is stable (i.e., reduction number of $I$ is $\leq 1$),
see \cite[Theorem 1]{GS}.  From the definition
it does not follow that if $I, J$ are $p_g$-ideals then the product $IJ$ is also a $p_g$ ideal.
However if $A$ is also analytically unramified with an infinite residue field then by a result of Rees;  product of two
$p_g$ ideals is $p_g$, see \cite[ 2.6]{R} (also see \cite[1.2]{P}).
\emph{We do not know that whether every normal domain of dimension two has a
$p_g$ ideal.}

 In \cite{P} we proved that if  $(A,\m)$ is  an excellent two dimensional normal domain containing a field $k \cong A/\m$ of characteristic zero,
 then there exists $p_g$ ideals in $A$. The technique in \cite{P} fails if $k$ has positive characteristic.
In this paper we prove
\begin{theorem}
 \label{existence-p}
 Let $(A,\m)$ be an excellent two dimensional normal domain containing a perfect field $k \cong A/\m$ of characteristic $p > 0$.
 Then there exists $p_g$ ideals in $A$.
\end{theorem}
To prove Theorem \ref{existence-p} we build on the techniques developed in \cite{P}. The main new technique are two spectral sequences discovered by Ellingsrud-Skjelbred; \cite[section 2]{ES}.

We now describe in brief the contents of this paper. In section two we discuss some preliminaries that we need.
In section three we discuss the two spectral sequences discovered by Ellingsrud-Skjelbred and give an application that we need. In section four we give a proof of Theorem \ref{existence-p}.

\section{preliminaries}
Throughout this section $(A,\m)$ is a Noetherian local ring of dimension two containing a perfect field $k \cong A/\m$.

Most of the following was poved in \cite{P}
\begin{lemma}\label{basic}
Let $(A,\m)$ be a Noetherian local ring containing a perfect field $k \cong A/\m$. Let $\ell$ be a finite extension of $k$. Set $B = A \otimes_k \ell$. Then we have the following
\begin{enumerate}[\rm (1)]
 \item $B$ is a finite flat $A$-module.
 \item $B$ is a Noetherian ring.
 \item $B$ is local with  maximal ideal $\m B$ and residue field isomorphic to $\ell$.
 \item $B$ contains $\ell$.
 \item $A$ is \CM \ (Gorenstein, regular) \ if and only if $B$ is \CM \ (Gorenstein, regular).
 \item  If $A$ is excellent then so is $B$.
 \item If $A$ is normal then so is $B$.
 \item If $A$ is excellent normal and $I$ is an integrally closed ideal in $A$ then $IB$ is an integrally closed ideal in $B$.
 \item If $\ell$ is a Galois extension of $k$ with Galois group $G$ then $G$ acts on $B$ (via $\sigma(a\otimes t) = a\otimes \sigma(t)$). Furthermore  $B^G = A$.
 \item
 If $A$ is \CM \ of dimension two, the natural map $H^2_\m(A) \rt H^2_\m (B)$ is an inclusion.
\end{enumerate}
\end{lemma}
\begin{proof}
For (1)-(8) see \cite[2.1]{P}.

 (9) It is clear that $G$ acts on $B$ (via the action described) and $A \subseteq B^G$. By normal basis theorem, cf., \cite[Chapter 6, Theorem 13.1]{Lang},  there exists $x \in \ell$ such that $\{ \sigma(x) \colon \sigma \in G \}$ is a basis of $\ell$ over $k$.
 A basis of $B$ as an $A$-module is $\{ 1 \otimes \sigma(x) \colon \sigma \in G \}$.

 Let $\xi \in B^G$. Let $\xi = \sum_\sigma a_\sigma(1 \otimes \sigma(x))$. Let $e$ be the identity in $G$. Let $\tau \in G$.
 Then notice
 \[
 \xi = \tau^{-1}\xi =  \sum_\sigma a_\sigma(1 \otimes \tau^{-1}\sigma(x))
 \]
 Comparing terms we get $a_{e} = a_\tau$ for all $\tau \in G$. So
 \begin{align*}
 \xi &= a_{e}\left( \sum_\sigma 1 \otimes \sigma(x)  \right), \\
  &= a_{e}(1 \otimes \sum_\sigma \sigma(x) ), \\
  &= a_{e}(r \otimes 1) \quad  \text{where} \ r = \sum_\sigma \sigma(x) \in k, \\
  &= a_{e}r(1\otimes 1) \in A.
 \end{align*}
 The result follows.

 (10) We have an exact sequence of finite dimensional $k$ vector spaces
 \[
 0 \rt k \rt \ell \rt V \rt 0.
 \]
 So we have an exact sequence of $A$-modules
 \[
 0 \rt A \rt B \rt A\otimes_k V \rt 0.
 \]
 We note that both $B$ and $A\otimes_k V$ are free $A$-modules. As $H^i_\m(A) = 0$ for $i < 2$ the result follows.
\end{proof}

\s \text{A construction:}
 Fix an algebraic closure $\ov{k}$ of $k$. We investigate properties of
$A\otimes_k \ov{k}$.

\s \label{limit} Let
$$\C_k = \{ E \mid E \ \text{is a finite extension of $k$ in $\ov{k}$} \}.$$
We note that  $\C_k$ is a directed system of fields with $\lim_{E \in \C_k} E  = \ov{k}$.
For $E \in \C_k$ set $A^E = A\otimes_k E$. Then by \ref{basic} $A^E$ is a finite flat extension of $A$. Also $A^E$ is local with maximal ideal $\m^E = \m A^E$.
Clearly $\{ A^E \}_{E \in \C_k}$ forms a directed system of local rings and we have $\lim_{E \in \C_k} A^E  = A\otimes_k \ov{k}$.
By \cite[Chap. 0. (10.3.13)]{EGA3} it follows that $A\otimes_k \ov{k}$ is a Noetherian local ring (say with maximal ideal $\m^{\ov{k}})$.
Note that we may consider $A^E$ as a subring of $A\otimes_k \ov{k}$. We have
$$ A\otimes_k \ov{k} = \bigcup_{E\in \C_k}A^E \quad \text{and} \quad \m^{\ov{k}} = \bigcup_{E \in \C_k} \m^E.$$
It follows that $\m (A\otimes_k \ov{k}) = \m^{\ov{k}}$. It is also clear that $A\otimes_k \ov{k}$ contains $\ov{k}$ and its residue field
is isomorphic to $\ov{k}$. The extension $A \rt A\otimes_k \ov{k}$ is flat with fiber $\cong \ov{k}$. In particular $\dim A\otimes_k \ov{k}$ is two.

\s\label{cofinal} Let $F \in \C_k$. Set
\[
 \C_F = \{ E \mid E \in \C_k,  E\supseteq F \}.
\]
Then $\C_F$ is cofinal in $\C_k$. So we have $\lim_{E \in \C_F} A^E = A\otimes_k \ov{k}$.
Also note that if $E \in \C_F$ then
\[
 A^E = A\otimes_k E = A\otimes_k F \otimes_F E = A^F \otimes_F E.
\]
It also follows that $\m^E = \m^FA^E$.

For a proof of the following result see \cite[3.3]{P}.
\begin{lemma}
 \label{excellent}
 If $A$ is excellent then so is $A\otimes_k \ov{k}$
\end{lemma}

The main properties of $A\otimes_k \ov{k}$ that we need is the following is summarised in the following result which is Theorem 3.4 in \cite{P}.
\begin{theorem}
 \label{eclair}
 (with hypotheses as above) Set $T = A\otimes_k {\ov{k}}$ and $\n = \m^{\ov{k}}$. We have
 \begin{enumerate}[\rm (1)]
  \item  $A$ is \CM \ (Gorenstein, regular) if and only if $T$ is \CM \ (Gorenstein, regular).
  \item If $A$ is a normal domain if and only if $T$ is a normal domain.
  \item Assume $A$ is  an excellent normal domain. Then we have
  \begin{enumerate}[\rm (a)]
   \item $I$ is integrally closed in $A$ if and only if $IT$ is integrally closed in $T$
   \item $I$ is a $p_g$ ideal in $A$ if and only if $IT$ is a $p_g$ ideal in $T$.
  \end{enumerate}
 \end{enumerate}
\end{theorem}

\section{ Ellingsrud-Skjelbred spectral sequences and an application}
In this section we describe the Ellingsrud-Skjelbred spectral sequences (we follow the exposition given in \cite[8.6]{Lorenz}). We also give an application which is crucial for us.

\s Let $S$ be a commutative Noetherian ring. Let $Mod(A)$ be the category of left $A$-modules.

(1) Let $G$ be a finite group. Let $S[G]$ be the group ring and let $Mod(S[G])$ be the category of left $S[G]$-modules. Let $(-)^G$ be the functor of $G$-fixed points. Let $H^n(G,-)$ be the $n^{th}$ right derived functor of
$(-)^G$.

(2) Let $\A$ be an ideal in $S$. Let $\Gamma_\A(-)$ be the torsion functor associated to $\A$.  Let $H^n_\A(-)$ be the $n^{th}$ right derived functor of
$\Gamma_\A(-)$. Usually $H^n_\A(-)$ is called the $n^{th}$ local cohomology functor of $A$ \wrt \ $\A$.

(3) If $M \in Mod(S[G])$ then note $\Gamma_\A(M) \in Mod(S[G])$.

\s Ellingsrud-Skjelbred spectral sequences  are constructed as follows: Consider the following sequence of functors
\[
(i) \quad   Mod(S[G]) \xrightarrow{(-)^G} Mod(S) \xrightarrow{\Gamma_\A} Mod(S),
\]
\[
(ii)   \quad   Mod(S[G]) \xrightarrow{\Gamma_\A} Mod(S[G]) \xrightarrow{(-)^G} Mod(S)
\]
We then notice
\begin{enumerate}[\rm (a)]
\item
The above compositions are equal.
\item
It is possible to use Grothendieck spectral sequence of composite of functors to both (i) and (ii) above; see
\cite[8.6.2]{Lorenz}.
\end{enumerate}
Following Ellingsrud-Skjelbred  we let $H^n_\A(G,-)$ denote the $n^{th}$ right derived functor of this composite
functor. So by (i) and (ii) we have two first quadrant spectral sequences for each $S[G]$-module $M$
\[
(\alpha)\colon \quad \quad  E_2^{p,q} = H^p_\A(H^q(G, M)) \Longrightarrow H^{p+g}_\A(G,M), \ \text{and}
\]
\[
(\beta)\colon \quad \quad \mathcal{E}_2^{p,q} = H^p(G, H^q_\A(M)) \Longrightarrow H^{p+g}_\A(G,M).
\]
\begin{remark}\label{grading}
(1) If $S$ is $\mathbb{N}$-graded ring then $S[G]$ is also $\mathbb{N}$-graded (with $\deg \sigma = 0$ for all $\sigma \in G$).

(2)  If $M$ is a finitely generated graded left $S[G]$-module then $H^n(G, M)$ are finitely generated graded $S$-module. This can be easily seen by taking a graded free resolution of $S$ consisting of finitely generated free  graded $S[G]$-modules.

(3) Ellingsrud-Skjelbred spectral sequences have an obvious graded analogue.
\end{remark}

\s\emph{Application:}\label{setup} \\
Setup: Let $(A,\m)$ be a two dimensional \CM \ local ring containing a field $k \cong A/\m$. Let $\ell$ be a finite Galois extension of $k$ with Galois group $G$. Set $B = A\otimes_k \ell$ and let $\n$ be maximal ideal of $B$. Let $G$ act on $B$ (as described in \ref{basic}(9)). Note $B^G = A$; see \ref{basic}(9). Let $I$ be an $\n$-primary ideal of $B$ which is $G$-invariant (i.e., $\sigma(I)= I$  for each $\sigma \in G$). Let $\R(I) = B[It]$ be the Rees algebra of $I$. Then note we have a natural action of $G$ on $\R(I)$. Let $\C= \R(I)^G$. By a result of E.Noether we have $\C$ is a graded finitely generated $A = \C_0$-algebra. Furthermore $\R(I)$ is a finite $\C$-module.
We prove
\begin{theorem}
\label{ES-app}
(with hypotheses as in \ref{setup}.) Let $\M$ be the graded maximal ideal of $\C$. If $\R(I)$ is \CM \ then
\begin{enumerate}[\rm (1)]
\item
$H^0_\M(\C) = H^1_\M(\C) = 0$.
\item
$H^2_\M(\C)$ has finite length and $H^2_\M(\C)_0 = 0$.
\item
For some $r > 0$ there exists an $\m$-primary ideal $J$ in $A$ such that $A[Jt] = \C^{<r>}$ is \CM.
\end{enumerate}
\end{theorem}
\begin{proof}
  We first consider the Ellingsrud-Skjelbred spectral sequence $(\beta)$ with $S = \C$ and $M = \R(I)$.  We note that as $\R(I)$ is a finite $\C$-module we get that $\sqrt{\M \R(I)}$ is the graded maximal ideal of $\R(I)$.
  As $\R(I)$ is \CM \ of dimension $3$ we get that $H^i_\M(\R(I)) = 0$ for $i \leq 2$. Furthermore by Grothendieck Vanishing theorem we get $H^i_\M(\R(I)) = 0$ for $i > 3$.
  Thus $(\beta)$ collapses at the second stage. In particular we have $H^{r}_\M(G, \R(I)) = 0$ for $r = 0, 1, 2$.

  Next we consider the Ellingsrud-Skjelbred spectral sequence $(\alpha)$ with $S = \C$ and $M = \R(I)$.

  (1) We have $E_2^{0,0} = H^0_\M(\C)$. Furthermore it is clear that $E_2^{0,0} = E_\infty^{0,0}$. As $E_\infty^{0,0}$ is a sub-quotient of $H^0_\M(G, \R(I)) = 0$ we get  $H^0_\M(\C) = 0$.

  We also have $E_2^{1,0} = H^1_\M(\C)$. Furthermore it is clear that $E_2^{1,0} = E_\infty^{1,0}$. As $E_\infty^{1,0}$ is a sub-quotient of $H^1_\M(G, \R(I)) = 0$ we get  $H^1_\M(\C) = 0$.

  (2) We have $ E_2^{0,1} = H^0_\M(H^1(G,\R))$ and $E^{2,0}_2 = H^2_\M(\C)$.
  Note we have an exact sequence
  \[
  0 \rt E_3^{0,1} \rt H^0_\M(H^1(G,\R(I))) \rt  H^2_\M(\C) \rt E_3^{2,0} \rt 0.
  \]
  Furthermore it is clear that $E_3^{0,1} = E_\infty^{0,1}$ and $E_3^{2,0} = E_\infty^{2,0} $. Furthermore $E_\infty^{0,1}$ and $E^\infty_{2,0} $ are sub-quotients of $H^1_\M(G, \R(I))$ and $H^2_\M(G, \R(I))$ which are zero. It follows that we have an graded isomorphism
  $$  H^0_\M(H^1(G,\R(I))) \cong  H^2_\M(\C).$$

  2(a) As $H^1(G, \R(I))$ is a finitely generated $\C$-module we get that $H^2_\M(\C)$ has finite length.

 2(b)  We note that we have a graded inclusion
 $$H^0_\M(H^1(G, \R(I))) \subseteq H^0_{\m \C}(H^1(G, \R(I))). $$
 Thus it suffices to prove $H^0_{\m \C}(H^1(G, \R(I)))_0 = 0$.
 Let $\mathbb{F}$ be a graded free resolution of $\C$ by finitely generated graded free $\C [G]$-modules with $\mathbb{F}_0 = \C [G]$. Then $H^1(G, \R(I))$ is the first cohomology module of the complex
 $\mathbb{W} = \Hom_{\C[G]}(\mathbb{F}, \R(I))$. Let $B^1$ and $Z^1$ be the module of first co-boundaries and first co-cycles of $\mathbb{W}$.

 Note  $Z^1$ is a submodule of $\Hom_{\C[G]}(\mathbb{F}_1, \R(I))$ which in turn is a submodule of $\Hom_{\C}(\mathbb{F}, \R(I)) \cong \R(I)^s$ for some $s$.
 In particular we have $H^0_{\m C}(Z^1) = 0$.

 We also have a graded exact sequence $0 \rt \C \rt \R(I) \rt B^1 \rt 0$. Note as $\m$ is an $A$-ideal (and as $\C_0 = A$ and $\R(I)_0 = B$) we get an exact sequence
 \[
 H^1_\m(B) \rt H^1_\m(B^1)_0 \rt H^2_\m(A) \rt H^2_\m(B).
 \]
 Note as $B$ is finite \CM  \ $A$-module of dimension $2$ we get $H^1_\m(B) = 0$. Also by \ref{basic}(10) the map $H^2_\m(A) \rt H^2_\m(B)$ is an inclusion. Thus  $H^1_\m(B^1)_0  = 0$.

 We have an exact sequence $0 \rt B^1 \rt Z^1 \rt H^1(G, \R(I)) \rt 0$. Taking cohomology we get $H^0_{\m \C}(H^1(G, \R(I)))_0 = 0$,  since $H^0_{\m \C}(Z^1) = H^1_{\m \C}(B^1)_0 = 0$.

 (3) We note that $\C_n = I^n\cap A$. As $B$ is a finite $A$-module we get that $\m B$ is $\n$-primary. As $I^n$ is $\n$-primary we get that $I^n$ will contain some power of $\m B$. As $B$ is flat $A$-module we get $\m^s B \cap A = \m^s$ for all $s \geq 1$. Thus $\C_n$ are $\m$-primary ideals of $A$. As $\C$ is Noetherian it follows that some Veronese $\C^{<m>}$ is standard graded. It follows that  $\C^{<m>} = A[\C_m t]$.

Local cohomology commutes with the Veronese functor.  By (2) it follows that $H^i_{\M^{<s>}}C^{<s>} = 0$ for all $s \geq s_0$ and $i = 0, 1, 2$. We take $r = s_0 m$. Then note that $\C^{<r>} = A[\C_r t]$ is \CM. Furthermore as discussed above $\C_r$ is $\m$-primary.
\end{proof}
\section{proof of Theorem \ref{existence-p}}
In this section we give
\begin{proof}[Proof of Theorem \ref{existence-p}] Let $T = A\otimes_k \ov{k}$. Let $\n$ be the maximal ideal of $T$.
 We note that $T$ is an excellent normal domain containing $\ov{k} \cong T/\n$ (see \ref{limit}, \ref{excellent} and \ref{eclair}(2)).
 By \cite[4.1]{OWY-1} there exists a $p_g$ ideal $J$ in $T$. By \ref{limit} we have $T = \bigcup_{E \in \C_k}A^E$. So there exists $F \in \C_k$ which contains a set of minimal generators
 of $J$. We may further assume  (by enlarging) that $F$ is Galois over $k$. Thus there exists ideal $W$ in $A^F$ with $WT = J$. By \ref{eclair}(3)(b) we get that
 $W$ is a $p_g$ ideal in $A^F$. Let $G$ be the Galois group of $F$ over $k$. Then $G$ acts on $A^F$ (via $\sigma(a\otimes f) = a\otimes \sigma(f)$). By \ref{basic}(9) we get  $(A^F)^G = A$. We also note that we have a natural $G$ action on $A^F[t]$ (fixing $t$) and clearly  its invariant ring is $A[t]$.
 Let $\sigma \in G$. It's action on $A^F[t]$ induces an isomorphism of between the Rees algebra's $\R(W)$ and $\R(\sigma(W))$. So $\sigma(W)$ is a $p_g$ ideal
 in $A^F$.
 As product of $p_g$ ideals is $p_g$ we get that $K = \prod_{\sigma \in G}\sigma(W)$ is a $p_g$ ideal in $A^F$. Note $K$ is $G$-invariant. So the $G$ action of $A^F[t]$
 restricts to a $G$-action on $\R(K)$. Set  $\C = \R(K)^G$. We note that $\C_n = K^n \cap A$ is an $\m$-primary integrally closed ideal for all $n \geq 1$.
 By Theorem \ref{ES-app} some Veronese of $\C^{<r>} = A[Jt]$ is \CM. As $J^n$ is integrally closed $\m$-primary for all $n$ we get that $A[Jt]$ is a \CM \ normal domain.
   Thus $J$ is a $p_g$ ideal in $A$.
\end{proof}

\end{document}